\documentclass[12pt,a4paper]{article}

\usepackage[cp1251]{inputenc}
\usepackage[russian]{babel}%

\usepackage[centertags]{amsmath}
\usepackage{amsfonts}
\usepackage{amssymb}
\usepackage{graphicx}
\usepackage{amsthm}
\begin{document}

\oddsidemargin=-0.1cm
\topmargin=-5mm

\newtheorem{lemma}{Lemma}[section]
\newtheorem{remark}{Remark}[section]
\newtheorem{theorem}{Theorem}[section]
\newtheorem{statement}{Section}[section]
\newtheorem{definition}{Definition}[section]
\newtheorem{proposition}{Предложение}[section]
\newtheorem{colorrary}{Colorrary}[section]
\newtheorem*{theorem*}{Threorem}
\newtheorem*{theorem*1}{Theorem \normalfont{on polyhedra}}
\newtheorem*{corollary}{Следствие}
\newtheorem*{theorem*2}{Theorem on index}
\newtheorem*{conj}{Conjecture}
\def\proof{\par\noindent{\bf Proof}. \ignorespaces}
\def\endproof{{\ \vbox{\hrule\hbox{%
   \vrule height1.3ex\hskip1.3ex\vrule}\hrule
  }}\vspace{2mm}\par}

\def\proofff{\par\noindent{\bf Proof} \ignorespaces}
\def\endproofff{{\ \vbox{\hrule\hbox{%
   \vrule height1.3ex\hskip1.3ex\vrule}\hrule
  }}\vspace{2mm}\par}

\date{   }
\title{\bf   Local Theory of $t$-bonded Sets
}
\author{M.~Bouniaev, N.~Dolbilin\\
mikhail.bouniaev@utrgv.edu, dolbilin@mi.ras.ru}


\maketitle

\renewcommand{\abstractname}{Abstract}
\renewcommand{\refname}{References}

\begin{abstract}
The local theory for regular  and multi-regular systems was
developed in the assumption  that these systems  are Delone sets,
or $(r,R)$-systems. The requirement for a set  to be a
$(r,R)$-system particularly implies that any two points in a
Delone  set can be connected by a sequence of points from the set
with sequel inter-point distances bounded by $2R$. In the
terminology we adopted  in this paper, it means that a Delone set
is a $2R$-bonded set. Meanwhile, there are crystals, e.g.
zeolites, whose atomic structure is  multi-regular  microporous
point set. In these structures there are  cavities that are
relatively large compared to the "length" of bonds  between
atomes. In other words,  the parameter  $R$ in this Delone set
significantly exceeds a natural link parameter. For a  better
description of such "microporous" structures it is worthwhile to
take  into consideration  a parameter that represents  atomic
bonds within the matter.    In the paper we generalize some
results of the local theory to the sets that  we called $t$-bonded
sets even without making an assumption that a set is a Delone set.

\end{abstract}

\section{Introduction}
In this section we  present basic definitions related to the mathematical
concept of crystal in the light of the local theory with the overarching goal to extend the
theory’s results by re-introducing the concept of t-bonded sets  [Definition 2.1]
and considering the class of sets that includes the Delone sets as a subclass.

The concept of t-bonded sets was briefly introduced by one of the
authors  in [1] under the name of $d$-connected sets (in Russian),
though it has not received  due consideration. In light of the
developments in the local theory for crystals that occurred since
1976 and  demands  in chemistry and crystallography, we believe
the local theory for $t$-bonded sets deserves to be developed.

The above mentioned definitions single  the family of crystals out
of the  family of more general sets, which  fulfils  the
requirements  for  point sets  to be uniformly not very close to
each other (see the  $r$-condition below), and relatively dense
(the  $R$-condition below). Sets with these conditions were
introduced and studied by B. Delone who called them $(r,
R)$-systems ([2],[3]).

\smallbreak \textbf{Definition 1.1} (Delone Set). Let
$\mathbb{R}^d$ be an Euclidean space and $r$ and $R$  some
positive numbers. A set $X\subset \mathbb{R}^d$,  is called a
\emph{Delone set} with parameters $r$ and $R$ (or $(r,R)$-system)
 if
\newline (i)    ($r$-condition):  any   open ball  $B(r)$  of radius $r$  has  at most  one point from $X$, and
\newline (ii)   ($R$-condition): any closed ball $B(R)$ of radius R has at least one point from  $X$.

\smallbreak \textbf {Remark 1.1}. The definition of a Delone set
requires the existence of numbers $r$ and $R$ with specified
properties. However, for the sake of shortening the theorems’
statements and proofs we included these two parameters into the
definition of a Delone set as a characteristic of the set in the
assumption that they exist. Even more, we chose  $r$ as the
supremum of all numbers  such that  set X  satisfies
{r}-condition, and $R$ as the infimum of the set of all numbers
that satisfy $R$-condition.

\smallbreak \textbf{Definition 1.2} (Regular System). A Delone set
$X$ is called a \emph{regular system} if for any two points $x$
and $y$ from $X$ there is a symmetry $s$ of $X$ such that $s(x) =
y$, i.e. if the symmetry group $Sym(X)$ acts transitively on $X$.

\smallbreak \textbf{Remark 1.2}. It follows immediately from
definition 1.2 that a regular point set $X$ is an orbit  $G\cdot
x$, where $x$ is a point from $X$, and $G$ is, generally speaking,
a subgroup of $Sym(X)$. We remind that $G$-orbit of $x$ is the set
$G\cdot x ={g(x)|g \in G}$.

\smallbreak Let $Iso(d)$ be the complete group of
all isometries of $\mathbb{R}^d$.

\smallbreak \textbf{Definition 1.3} (Discrete Group). A group
$G\subset Iso(d) $ is called a \emph{discrete subgroup}, if the
orbit $G\cdot x$ of any point $x\in\mathbb{R}^d$ is a discrete
subset of $\mathbb{R}^d$.

\smallbreak \textbf{Definition 1.4} (Fundamental Domain). Let $G$
be  a discrete subgroup of $Iso(d)$. We call the closed domain
$F(G)$ in  $\mathbb{R}^d$ a \emph{fundamental domain} of group $G$
if:
\newline(i)  for any point $x$ from $\mathbb{R}^d$, the
intersection of $F(G)$ and the orbit $G\cdot x$ is not empty;
\newline (ii) for any point $x$ from $\mathbb{R}^d$, the
interior of $F(G)$ contains  at  most one  point from $G\cdot x$.

\smallbreak \textbf {Remark 1.3.} For a discrete group $G$ a
fundamental domain does exist.  It suffices to take an orbit
$G\cdot x$ of a non-fixed point $x$ with respect to $G$  and
construct the Voronoi tessellation for the $G\cdot x$. The Voronoi
domain is a fundamental domain of the group. A fundamental domain
can be chosen in a non-unique way, sometimes it can be unbounded.

 \smallbreak \textbf{Definition 1.5.} [Crystallographic Group]. Let $Iso(d)$ be the complete group of
all isometries of  Euclidean $d$-space $\mathbb{R}^d$ . A subgroup
$G$ of the group $Iso(d)$ is called \emph{crystallographic} if any orbit
$G\cdot x$ is a discrete set, and the fundamental domain of $G$ is
compact.

Important  results for crystallographic groups were obtained in
[4],[5].

\smallbreak \textbf {Statement 1.1.} \emph{A Delone set $X$ is a
regular system  if and only if there is a crystallographic
group $G$ such that $X$ is a $G$-orbit of some point $x$.}

\smallbreak E.S. Fedorov defined  crystal as a finite union of
regular point sets [6].

\smallbreak \textbf{Definition 1.6.} [ Crystal]. We say that a  subset $X$ of
$\mathbb{R}^d$ is a \emph{crystal} if $X$ is the $G$-orbit of a
finite set $X_0=\{x_1,...,x_k\}$,
 i.e. $X=\cup_{i=1}^k G\cdot x_i$.

 \smallbreak Thus a crystal can be regarded as  a union of  orbits of several points  with respect to the same
 crystallographic group $G$.

\smallbreak
The main goal of the local theory for crystals is to develop a sound mathematical
theory and methodology that would serve as a model of crystalline structure and
formation from the pairwise identity of local arrangements around each atom.
However, before 1970s, there were neither formal  statements that used mathematical language and concepts,
 nor  rigorously proven results in this regard until B. Delone and R. Galiulin formulated the  problem,
 and Delone’s students N. Dolbilin, M. Stogrin,  and others (see for instance, [7]-[10])
 developed a mathematically sound  local theory of crystals.

\smallbreak We would like to point out
that Delone sets have the following property that plays a significant role in most proofs of the local theory of crystals.

\smallbreak \textbf {Statement 1.2.} \emph{For a Delone set $X$ and for any two points $x$ and $x'\in X$
there is a finite sequence of points from
$X$ $x=x_0,x_1,\ldots ,x_m=x'$ such that $|x_{i-1}x_i|\leq 2R$,
$i\in [1,m]$.}

\smallbreak We call such  sequence a $2R$-\emph{chain} and
denote it as $[x,...,x']$. We  call each closed interval
$[x_{i-1},x_i]$ a \emph{link} of the $2R$-\emph{chain}.

\smallbreak Following the terminology of Definition 2.1,  that
the next section starts with,  we can say that any Delone set is a
$2R$-bonded set.

\smallbreak For Delone sets we can even make a stronger statement than  statement 1.2.
In fact, the following statement that is proved in the next section  holds
true. For any Delone $(r,R)$-set $X$ there is such $\varepsilon = \varepsilon(r,R)>0$ that for
any two points $x$ and $x'\in X$ there is $(2R-\varepsilon)$-\emph{chain}  $[x, ...,x']$ [Statement 2.1].

\smallbreak
For some Delone sets the  value of $\varepsilon$ is very
small, and therefore the length of links is bounded from above by
an upper bound close  to $2R$. However, there are   many
crystalline structures,  e.g. for zeolites, such  that they are
presented as  $(r,R)$-sets, and  any two points of
the structure could be connected by a chain with the links which
are significantly smaller than $2R$, i.e. the parameter $t$ [Definition 2.1]  is
significantly smaller than $2R$.

 \smallbreak We can also note that in the proofs of theorems in  the local
theory that  use $2R$-linkage of Delone sets, the length of links in a chain that connects the  given two
points of  $X\subset \mathbb{R}^d$  is not essential,
but what is essential is that  any
two points of $X$ could be connected by a chain  with the links'
length not greater than the fixed number $t$, ($t$ depends on the
set $X$). The size of a local region that determines global
properties of the set $X$ also depends on the value of $t$. The
lesser the value of $t$,  the smaller region could be considered. In
this respect,  the value of $t=2R$ in many structures seems to be
unnecessary too large, though for some $X$ the value of $t$ can be
very close to $2R$.

 \smallbreak On the other hand,  in the assumption that $X$ is a Delone $(r,R)$-set, if the parameter $t$ is
significantly  smaller than $2R$, local conditions expressed in of $R$ happen to be not very efficient.
 We believe that local conditions that determine global properties of a set
could be found in terms of $t$, the parameter that shows the lengths of
chains' links that any two points of the set could be connected with.

 \smallbreak All these observations inspired us
to develop the  local theory for the  $t$-bonded sets. In this theory we
 do not assume (unless it is stated as a premise),
 that the  set $X$ under consideration is a Delone
set, and therefore these sets do not possess some properties that
were used in developing the local theory for
Delone sets.

\smallbreak
\section {Definitions and Main Results}
As we have already mentioned, in this  paper, unless stated differently,
we consider subsets (that we denote $X$,$Y$,$Z$, ...etc.) of $d$-space  $\mathbb{R}^d$
that are  uniformly discrete point sets, i.e. $X$ is a set such that there exists $r>0$ such that any ball
$B(r)\subset \mathbb{R}^d$ contains at most one point from $X$
(condition (i) in definition 1.1  of a Delone set). Thus $X$ fulfils
just the $r$-condition and, generally speaking,   it is not a Delone
set. However, in this  paper we reserve letter  $r$ for
the parameter in the $r$-condition and letter R for the parameter
for  the  $R$-condition respectfully in the definition of a
Delone $(r,R)$-set. Like  for a Delone set  we'll choose  $r$ as the
supremum of all numbers  such that the set X  satisfies the
{r}-condition, and $R$ as the infimum of the set of all numbers
that satisfy the $R$-condition.

\smallbreak \textbf{Definition 2.1.} [$t$-bonded  Set]. A set
$X\subset \mathbb{R}^d$ is said to be a \emph{$t$-bonded set}  in
$\mathbb{R}^d$, or just  \emph{$t$-bonded set}  in $\mathbb{R}^d$,
where $t$ is some positive number if:
\\ (1) \emph{aff} $X =\mathbb{R}^d$, where \emph{aff}$X$ stands for affine hull of $X$
\\ (2) For any two points $x$ and $x'\in X$ there is a finite sequence of points from
$X$ $x=x_0,x_1,\ldots ,x_m=x'$ such that $|x_{i-1}x_i|\leq t$,
$i\in [1,m]$. We will call the sequence a $t$-\emph{chain} and
denote it as $[x,...,x']$. Each closed interval $[x_{i-1},x_i]$ will be called
a \emph{link} of the $t$-\emph{chain}.

\smallbreak \textbf{Statement 2.1.} \emph{For any Delone
$(r,R)$-set $X$ there is such $\delta = \delta(r,R)>0$ that for
any two points $x$ and $x'\in X$ there is $t$-\emph{chain}  $[x,
...,x']$ with all links  that are no longer  than $2R-\delta$, i.e
any Delone set is a $t$-bonded set, where the parameter $t$ can be chosen
less than $2R$.}

\begin{proof} First of all, we'll show  that  due to the $R$-property of a Delone
set $X$ for any two points $x$ and $x'\in X$ there is a $2R$-chain
$[x,\ldots , x']$. Let  $B_z(R)$ denote a ball with radius $R$ and
the center $z$ such that $z\in [xx']$ and $|xz|=R$. The ball
$B_z(R)$ contains, generally speaking, several  points $y,y_1,
\ldots$ from $X$ different from $x$. Each of them fulfils
$0<|xy|\leq 2R$ and $|yx'|<|xx'|$. If the point
$x'$ is among these points $y,y_1,\ldots$,  the required chain is already complete. If among  points
$y,y_1,\ldots$ there is no $x'$, each of them can be chosen as the
first point $x_1$ of the $2R$-chain $[x,x_1,\ldots, x']$ being constructed.
Let us  take  a ball $B_{z_2}(R)$ with the radius $R$
and the center $z_2\in [x_1,x']$. In $B_{z_2}(R)$ there are again
finitely many points such that each of them can be chosen as the
second point of $2R$-chain $[x,x_1,x_2,\ldots ,
x']$ being constructed. Applying this argument again and again we come to the
$2R$-chain $[x,x_1,x_2,\ldots, x']$.

Proving the  existence of a $(2R-\delta)$-chain
will require additional arguments. First, it is easy to show that if on a circle with a radius no
greater than $R$ there are at least three points
$y_1,y_2,y_3,\ldots $ such that for all the  points $|y_iy_j|\geq
2r$, then any two of these  points  can be connected by a
$t$-chain of the points $y_1,y_2,y_3,...$, where
$$t\leq\sqrt{4R^2-r^2}=2R(1-\sqrt{1-(r/2R)^2}).$$

Now we construct a
   $t$-chain $[x,\ldots ,x']$ which connects $x$ to $x'\in X$. We start with a
 'small' ball $B_{z}(\rho)$, where the center $z\in [xx']$ and $|xz|=\rho$, i.e. $x$ is on the boundary of the ball.
If $\rho<r$ the ball $B_z(\rho)$  contains  only one point from
$X$. It is the  point $x$. Now we  keep  shifting the
ball's center $z$ from  the point $x$  along the segment $[xx']$.
At the same time we keep increasing the radius $\rho$ of the  ball
so that the boundary of the  inflated ball passes through the point
$x$. Due to the $R$ and $r$-conditions of a Delone set at some
 point in time the inflated ball $B_z(\rho)$ 'catches up' a new point $x_1\in
 X$. It is obvious that the radius of the ball does not exceed
 $R$.

 We can assume that there are no other points  from $X$ on the  boundary $\partial B_z(\rho)$.
  Therefore  we can continue to inflate the
 ball and at the same time keep points $x$ and $x_1$
 on the ball's  boundary,  until the ball is constructed  with at least three points from X on its  boundary.
 Since there are no points from $X$ inside the ball $B_z(\rho)$, its
 radius does not exceed $R$.
 Furthermore, any three points on a sphere
 are necessarily non-collinear. They have to be on a circle
 which is a plane section of the boundary sphere of the ball.
 In other words, we have at least  three points $x, x_1$ and
 $x_1\in X$ on a circle of the radius not greater than $R$.
 Due to the above mentioned remark on three points on a circle
 in $X$ there is a $t$-chain that connects $x$ to $x_1$. This chain
 $x,...,x_1$  can be chosen as the starting  fragment of the
  $[x,\ldots, x']$ to be constructed. We emphasize that this
 fragment is already not monotonic. Indeed, though $|x'x_1|<|x'x|$,
 other intermediate points of the fragment can be further from
 $x'$ than points $x$ and $x_1$. If the point $x_1$ differs from
 $x'$,  we can apply the same argument and construct a new fragment
  $[x_1,\ldots, x_2]$ of the  $t$-chain, where $|x_2x'|<|x_1,x'|$.
Continuing the process  we will construct the $t$-chain $[x\ldots, x']$.
\end{proof}

 \smallbreak
It is clear that $t\geq 2r$. For some Delone sets the value of $t$
can be chosen significantly less than $2R$. For instance, if
$X=\mathbb{Z}^d$ is the  cubic lattice, then $t$ can be chosen as
the edge length of the cube and $2R=t\sqrt{d}$.

\smallbreak
In the local theory the concept of cluster  plays a significant
role and there could be different approaches to this concept.
In this paper we consider a version of the cluster
which  was mainly used in the local theory for Delone sets.
 We note that because the concept of  cluster  we adopt
 in this paper is the same for both Delone sets and t-sets,
 the concepts of clusters' equivalence and cluster's group  of symmetries $S_x(\rho)$ are also the same.

 \smallbreak
 \textbf{Definition 2.2. } [Cluster]. Let $\rho > 0$, a $\rho$-\emph{cluster}
$C_x(\rho)$ centered at point $x\in X$ be defined as a set of all
points $x'\in X$ such that $|xx'|\leq \rho$, i.e.
$$C_x(\rho) = X\cap B_x(\rho)$$

 \smallbreak
 Local conditions for a set to be  regular  are normally
expressed in terms of $R$. However,   in  case of $t$-bonded sets, since
we do  not require the $R$-condition, there are some interesting
properties of a set that could not be expressed in terms of $R$.
Some properties of a set fail to be true if we try to replace
$R$ with other parameters that seem to play a similar role.
For example, it is natural that parameter $t$ for a
$t$-bonded set has a role similar to that of the parameter $2R$ in a
Delone set.  However, there is an important difference between
$2R$-clusters in  case of Delone sets and $t$-clusters in  case
of $t$-bonded sets. As is known,  in a  Delone set any  $2R$-cluster has
its affine hull of  full dimension $d$. In a $t$-bonded set $X\subset
\mathbb{R}^d$ which is not a Delone set, an analogous statement
that the dimension of the affine hull of the $t$-cluster (rank of
a $t$-cluster) is also equal to $d$ fails to be true. Moreover,
there are Delone sets $X$, and even regular systems, in which
parameter $t$ is significantly smaller than $2R$ and the affine hulls
of all clusters are two-dimensional, though  Delone sets are
full-dimensional. We'll start  section 3 with an example
of a t-bonded set that is a regular system, therefore a Delone
set, though the affine hull of a t-cluster in not equal to $d$. In the
same section we will introduce some conditions that guarantee that
a given cluster   has rank $d$, i.e. the dimension of the affine hull
of the cluster equals $d$  [Theorem 3.1].

The same conditions will guarantee that the rank of a cluster would not increase with the growth of its radius. This fact will play an important role
for the proofs  in the local theory for $t$-bonded sets.

 \smallbreak
\textbf{Statement 2.2.}  \emph{Given $t$-bonded set $X$, $\rho >0$, let
$X\setminus C_x(\rho)\neq \emptyset$. Then there is a  point
$x'\in X$ such that
  $x'\in C_x(\rho +t)\setminus C_x(\rho)$, and it is linked
  to the center $x$ by a $t$-chain  contained in $C_x(\rho +t)$.}

 \smallbreak
It should be noted  that in the  spherical layer $B_x(\rho
+t)\setminus B_x(\rho )$, generally speaking,  there could also be
points $x''\in X$, that are linked to $x$ just by a 'long'
$t$-chain starting at the center $x$ of the cluster  $C_x(\rho + t)$, leaving it,
and then coming back to the cluster $C_x(\rho +t)$ to get eventually connected to $x''$.

The  concept of  cluster was used to develop the local theory
for crystals that by definition are Delone sets. Here we would
like  to mention that in  case of $t$-bonded sets,  two more,
different from the traditional,  concepts of cluster might
play  an interesting role - combinatorial clusters and "mixed"
clusters. However,  development  of the local theory of crystals
that uses a different concept of clusters requires a separate
discussion that goes well beyond the goals of this paper.

\smallbreak In   section 4 we study the structure of  cluster's
group symmetries that plays a pivotal role in the local theory of
t-sets (and crystals). At the end of the section the $t$-extension
theorem is proved,  that gives sufficient conditions to extend $
\rho_0$-cluster isometry to the $(\rho_0+t)$-cluster isometry.

\smallbreak \textbf{Definition 2.4. } [Cluster Equivalency]. Given
$t$-bonded set $X$ in $\mathbb{R}^d$, $\rho > 0$, two points $x\in
X$ and $x'\in X$, we say that the $\rho$-cluster $C_x(\rho)$ is
\emph{equivalent} to  the $\rho$-cluster $C_{x'}(\rho)$, if there
is a space isometry $g$ of $\mathbb{R}^d$, such that $g(x)=x'$ and
$g(C_x(\rho)) = C_{x'}(\rho)$.

\smallbreak
In section 5 and 6  we prove two theorems [Theorem 5.1 and Theorem 6.1]
for t-sets that are similar to the Criterion for Regular (Delone)
Systems (see, e.g. [7], [11], [12], [13]) and Criterion for Crystals (see, e.g. [13], [14], [15]). Though the
statements of the theorems are almost identical for both Delone
sets and t-sets the  main challenge of the  proofs is related to
the rank of the clusters, which as we have already mentioned is $d$
for 2R-clusters in Delone sets, however, in case of t-sets it may not
be equal to  $d$ for  $\rho$-clusters when  $\rho$ is equal to
$t$. The  cluster's rank naturally affects the
structure of the group $S_x(\rho)$ of the cluster's symmetries.
 The statements of both theorems, as well as their proofs
depend on the concept of cluster  counting function that we define below.

 \smallbreak It is clear that with a given $\rho > 0$, the relation of clusters to be equivalent as  defined above,  is
 an equivalence relation   on a set of all $\rho$-clusters in X.   Therefore, the set
 of all $\rho$-clusters in X could be presented as a disjoint union of equivalence classes.

 \smallbreak
 \textbf{Definition 2.5.} [Cluster Counting Function] For a given
$t$-bonded set $X$ in $\mathbb{R}^d$ and  $\rho > 0$,  the \emph{cluster
counting function} $N(\rho)$ is defined as the cardinal number of
the set of equivalence classes of $\rho$-clusters in $X$ provided
the cardinal number of equivalence clusters is finite.

 \smallbreak
\textbf{Definition 2.6.}  [Set of Finite Type]. Set $X$  is said
to be \emph{of  finite type} if  the cardinal number  $N(\rho)$ is
finite for any $\rho>0$.

 \smallbreak
\textbf{Statement 2.3.} \emph{For a set $X$ of finite type the
cluster counting function $N(\rho)$ is defined and finite for any
$\rho \geq 0$; it is a positive, piecewise constant,  integer
valued, monotonically non-decreasing, and continuous from the left
function.}

\smallbreak Statement 2.3 is true for both Delone sets and t-sets. In  section 7 we  discuss
t-sets of finite type,  and we  show how different Delone sets and t-sets are, as far as the property
"to be a set of  finite type" is concerned.

\smallbreak
In  section 7 we also included the proof of  an anecdotal fact [Statement 7.1]  that in
case of Delone sets there exists a local condition
of  a Delone set $X$ to be a set of finite type. If for a Delone
set $X$ the counting function $N(\rho)$ is finite for $\rho=2R$,
then it is finite for all $\rho>0$. In case of $t$-bonded sets the
situation is quite different. We will prove the following theorem.

\smallbreak \textbf{Theorem  7.1.} \emph{Given $t>0$, for an
arbitrarily large  $k>0$ there is a $t$-bonded set $X\subset
\mathbb{R}^2$ such that $N(kt)<\infty$, but  for any
$\varepsilon>0$ $N(kt+\varepsilon)=\infty$.}

An example of such $t$-bonded set will be presented  in the proof of theorem 7.1.

\smallbreak
\section{ The Rank of a Cluster}
In this section we will discuss the \emph{rank} of a cluster, i.e.
the dimension of the affine hull  of a cluster. As we have already
stated in the previous section, given a point set  $X\subset
\mathbb{R}^d$, there is an important difference between
$2R$-clusters in  case of Delone sets and $t$-clusters in  case of
$t$-bonded sets. As is well-known  for  Delone sets,  a $2R$-cluster
for any point $x\in  X$ has the rank equal to $d$. Though  $2R$
for Delone sets  and $t$ for $t$-bonded sets play  similar roles in local
theories for Delone sets and $t$-bonded sets respectfully, as is shown
below,  an analogous to the previous statement for $t$-bonded sets fails
to be true.

\smallbreak \textbf{Statement 3.1.} \emph{There are Delone sets
$X$ of rank 3 and  even regular systems  in which the parameter
$t$ is significantly smaller that $2R$,  and all $t$-clusters have
rank 2, although  Delone sets have rank 3.}

\smallbreak \textbf{Example.} Let $\Lambda\subset \mathbb{R}^3$ be
a lattice of rank 3 constructed on the orthonormal basis,  and
$X:=\Lambda\setminus (2\Lambda\cup (2\Lambda +(1,1,1)))$. Then $X$
is a Delone set with parameters $r=1/2$ and $R=1$. Since all
$2R$-clusters in $X$ are centrally symmetrical and pairwise
equivalent,  then $X$ is a regular system ([15], [16],[17]). On
the other hand, $X$ is a $t$-bonded set where $t=1<2=2R$. Each
$t$-cluster $C_x(1)$ is a cross of rank 2. However,
the set $X$ is a regular system of rank 3.

\smallbreak
Though the situation with clusters' ranks in case of Delone sets and $t$-bonded sets is different,
 the  following statements (lemma 3.1, theorems 3.1 and 3.2)
 on ranks of $\rho$-clusters  hold true. To shorten the notation,  we'll use $d_x(\rho):=dim (\emph{aff}C_x(\rho))$
for the rank of  the cluster $C_x(\rho))$.

In all discussions below  $\Pi^n$ stands for the $n$-dimensional plane that is the affine hull of a $t$-bonded set $X$, $n\leq d$.

\smallbreak
\begin{lemma} Let $X\subset \mathbb{R}^d $ be  a $t$-bonded set, $\rho$  a positive real number, and   $x$,  $x'$ two  points from $X$ such that $|xx'|\leq t$,
 and  the following conditions hold true,
$$
d_x(\rho )= d_x(\rho+t)\, \hbox{ and }\,  d_{x'}(\rho )=
d_{x'}(\rho+t) \eqno(3.1)
$$

Then aff$C_{x'}(\rho)$=aff$C_{x'}(\rho+t)$=aff$C_x(\rho+t)$=aff$C_x(\rho+t)$.
\end{lemma}

\begin{proof} Since $|xx'|\leq t$, it follows that $C_x(\rho) \subset C_{x'}(\rho+t))$ and $C_{x'}(\rho) \subset C_x(\rho+t))$.
Hence,\emph{aff}$C_x(\rho)$ $\subset$  \emph{aff}$C_{x'}(\rho+t)$ and\emph{aff}$C_x(\rho)$ $\subset$ \emph{aff}$C_{x'}(\rho+t)$. From the premises
of the lemma  $d_x(\rho )=d_x(\rho+t)$ and $d_{x'}(\rho )=d_{x'}(\rho+t)$,
it follows that \emph{aff}$C_x(\rho+t)$ =\emph{aff}$C_x(\rho)$ $\subset$ aff$C_{x'}(\rho+t)$,
and \emph{aff}$C_{x'}(\rho+t)$=\emph{ aff}$C_{x'}(\rho)$ $\subset$ aff$C_x(\rho+t)$. Therefore, \emph{aff}$C_{x'}(\rho)$=\emph{aff}$C_{x'}(\rho+t)$=\emph{aff}$C_x(\rho+t)$=\emph{aff}$C_x(\rho+t)$
\end{proof}

\smallbreak
\begin{theorem}
Let $X\subset \mathbb{R}^d $ be  a $t$-bonded set, and there is some
$\rho>0$ such that for any point  $x$ from $X$, the following condition holds true,
$$
d_x(\rho )= d_x(\rho+t)=n(x) \eqno(3.2)
$$
Then $n(x)\equiv d$ and $\forall x\in X$
aff$C_x(\rho)$=aff$(X)=\mathbb{R}^d$.
\end{theorem}

\begin{proof} Let us take a point $x_0\in X$ and denote
\emph{aff}$C_{x_0}(\rho):=\Pi^n$ where $n=d_{x_0}(\rho)$. We'll show
that any point $x\in X$ belongs to $\Pi^n$, and therefore
$\Pi^n=\mathbb{R}^d$.

Let $[x_0,x_1...x_i...x_n=x]$ be a  $t$-chain that  connects $x_0$
and  $x$. By the $t$-bonded set definition for any  points $x$ and
$x'$ there exists a $t$-chain that connects $x$ and
$x'$. Since for any $i$, $0\leq i\leq {n-1}$
the distance $|x_i x_{i+1}|\leq t$, if follows from lemma3.1, that for any $i$, $0\leq i\leq {n-1}$,
\emph{aff}$C_{x_i}(\rho)$=\emph{aff}$C_{x_{i+1}}(\rho)$. Therefore,
\emph{aff}$C_{x_0}(\rho)$=\emph{aff}$C_x(\rho)$. We proved that any point $x\in
X$, belongs to \emph{aff}$C_{x_0}(\rho)$. Meanwhile, \emph{aff}$C_{x_0}(\rho)$
$\subset$ \emph{aff}$X$, and therefore
\emph{aff}$C_{x_0}(\rho))$=\emph{aff}$(X)=\mathbb{R}^d$ for any $x_0 \in X$.
\end{proof}

\smallbreak
\begin{theorem}
Let $X\subset \mathbb{R}^d $ be  a $t$-bonded set, such that  for every given
$\rho\leq t\cdot (d-1)$ the ranks of all $\rho$-clusters are equal
($d_x(\rho)=d(\rho)$, $\forall x\in X$ ). Then, for any $\rho'\geq d\cdot t$ and any $x\in X$, the rank $d_x(\rho')\equiv d$.
\end{theorem}

\begin{proof} Let us consider two cases  $d(t)=1$ and  $d(t)\geq 2$.

Case $d(t)=1$. $C_x(t)\subset L_0$. We'll show that $X\subset L$, i.e.
any $x'\in X$ is on the $L_0$, In fact, let us connect $x$ to $x'$
by a $t$-chain $(x=x_0, \ldots , x_k=x')$. Let $L_i$ be a line of
cluster  $C_{x_i}(t)$ for $i\in [0,k]$. Since for any $i\in[1,k]$
$C_{x_{i-1}}(t)\cap C_{x_i}(t)\supset \{x_{i-1}, x_i\}$, it implies that
the lines $L_{i-1}$ and $L_i$ of neighboring clusters are passing through both
these points: $x_{i-1}, x_i\in L_{i-1}$ and  $x_{i-1}, x_i\in
L_i, \, i\in [1,k]$. Hence,  $L_{i-1}=L_i$. Since the last identity is true
for all $i\in [1,k]$, we conclude that   all lines $L_i$ coincide with the
line $L_0$. Therefore,  $L_0=L_1=\cdots =L_k\ni x'$. We showed  that any point $x'\in L_0$.

We proved that $d(t)=1$,  implies  $d(\rho')=1$ for all $\rho'\geq d\cdot t$. This concludes the proof of the case $d(t)=1$.

Let us assume now that  $d(t)\geq 2$, and consider the function $d(\rho)$ defined at $[1,d\cdot t]$ and its values at
points $kt$: $d(t)=2\leq d(2t)\leq  d(3t)\leq \ldots \leq d(d\cdot t)$.
If all these  inequalities are strong,  then  $d(d\cdot t)=d+1$ which  is
impossible since $X\subset \mathbb{R}^d$. Therefore,  there is $k$,
$1\leq k\leq d-1$ such that $d(k\cdot t)=d((k+1)\cdot t)$. Hence,  by  Theorem 3.1,
$d(k\cdot t)=d((k+1)\cdot t)=d(\rho ')=d$ for all $\rho '\geq d\cdot t$.
\end{proof}

\smallbreak
\textbf {Remark.} We proved that under the conditions of  Theorem 3.2
 stabilization of the rank of any cluster definitely occurs when $\rho \geq d\cdot t$. However, for some sets  it might occur  even if $\rho \le d\cdot t$.

\smallbreak
\section{Symmetry of Clusters}
In this paragraph we again assume that X is a t-set in $\mathbb{R}^d$ which  by definition implies
 \emph{aff}$X =\mathbb{R}^d$.
Let us denote by $O(x,d)$ a group of all isometries of space
$\mathbb{R}^d$ which leave $x\in \mathbb{R}^d$ fixed.

 \smallbreak
 \textbf{Definition 4.1.} [The Symmetry of a Cluster.]
Assume
$x\in X$, then an isometry $\tau\in O(x,d)$ is called a
\emph{symmetry} of the cluster $C_x(\rho)$ if
$\tau(C_x(\rho))=C_x(\rho)$.

 \smallbreak
We want to emphasize that since
$\tau\in O(x,d)$,  any symmetry $\tau$  of a cluster leaves its center $x$ fixed.
We denote by $S_x(\rho)$ a group of all symmetries $\tau$ of the cluster
$C_x(\rho)$.

\smallbreak Now let\emph{aff}$C_x(\rho))=\Pi_x^n$ where $\Pi^n_x$ is an
$n$-dimensional affine plane, $x\in \Pi^n_x$, and $n \leq d$. We
denote by $\overline{S}_x(\rho)$ a  group of all isometries from
$O(x,n)$ that leave invariant the plane $\Pi_x^n$ and the cluster
$C_x(\rho)$.

If $n=d$, then   $\overline{S}_x(\rho)=S_x(\rho)$. Let  $n<d$, then
we denote the affine hull of $C_x(\rho)$ by $\Pi_x^n$, and the complementary orthogonal $(d-n)$-plane passing through
$x$ by $Q_x^{d-n}$. Let  $s\in S_x(\rho)$ be a symmetry of the $\rho $-cluster
$C_x(\rho)$. It is clear that any such symmetry is an orthogonal
transformation of  the $d$-space  which is a product of the
transformation $\bar s\in \Pi^n_x$ and of an arbitrary
transformation $g\in O(x,d-n)$ of the complementary plane. The
following lemma summarizes some facts on the cluster group.

\smallbreak
\begin{lemma} The following statements hold true:
\\ (1) If aff$C_x(\rho))=\Pi_x^n$ and $n<d$, then $S_x(\rho)= \overline{S}_x(\rho)\bigoplus O(x,d-n)$, where  $\overline{S}_x(\rho)
\subset O(x,n)$, and $O(x,d-n)$ is the  full group of isometries
 of the plane $Q_x^{d-n}$ complementary to  $\Pi_x^n$ and passing through the point $x$.
\\ (2) The group $\overline{S}_x(\rho)$ is a finite subgroup of $O(x,n)$ . Particularly,  if aff$C_x(\rho)=\mathbb{R}^d$,
then group $S_x(\rho)=\overline{S}_x(\rho)$  is a finite subgroup of $O(x,d)$;
\\ (3)  The group $S_x(\rho)$ is finite if and only if aff$C_x(\rho)=\mathbb{R}^d$ or aff$C_x(\rho)=\mathbb{R}^{d-1}$.
\end{lemma}

\begin{proof}
 (1) Any symmetry $\tau$ from $S_x(\rho)$ can be represented as
a  product of two isometries $f\cdot g$ where $f\in
\overline{S}_x(\rho)$ and $g\in O(x,d-n)$ is an arbitrary isometry
that operates in the$(d-n)$-plane $\Pi_x^{d-n}$ plane,
complementary to the plane $\Pi_x^n$,   and leaves the center $x$
fixed. At the same time, the product of any symmetry from the
group $\overline{S}_x(\rho)$ and any symmetry  $g\in O(x,d-n)$, is
a symmetry from $S_x(\rho)$ . Therefore $S_x(\rho)=
\overline{S}_x(\rho)\bigoplus O(x,d-n)$.

 (2) Let us prove first that if \emph{aff}$C_x(\rho))=\Pi_x^n$ and
$n<d$, then $\overline{S}_x(\rho)$ is a  finite subgroup of
$O(x,n)$. According to (1),  $S_x(\rho)=
\overline{S}_x(\rho)\bigoplus O(x,d-n)$ where
$\overline{S}_x(\rho)$ is a group of all symmetries $f\in O(x,n)$
that operates on $C_x(\rho)$ as a subset of $\Pi_x^n$. Since
$C_x(\rho)$ is a finite set and any point in $\mathbb{R}^n$ is an
affine combination of points from $C_x(\rho)$, we conclude  that
any  $\tau$ from  $\overline{S}_x(\rho)$ is completely determined
by its values on the points $x$  from the finite set $C_x(\rho)$.
Therefore, $\overline{S}_x(\rho)$ is a finite subgroup of
$O(x,n)$. If $n=d$ then $\overline{S}_x(\rho)=S_x(\rho)$,  hence,
$S_x(\rho)=\overline{S}_x(\rho)$ is a finite subgroup of $O(x,d)$.

 (3) It follows immediately from the second part of  the lemma that the condition \emph{aff}$C_x(\rho)=\mathbb{R}^d$ is  sufficient
for the group $S_x(\rho)$ to be  finite. It follows from (1) that  $n<d$ implies
then $S_x(\rho)= \overline{S}_x(\rho)\bigoplus O(x,d-n)$. If $d-n>1$, then  $O(x,d-n)$ is an infinite group, and therefore $S_x(\rho)$ is infinite.
If $d-n=1$, then $O(x,d-n)$ is a finite group and $S_x(\rho)= \overline{S}_x(\rho)\bigoplus O(x,d-n)$ is a finite group as a product of two finite groups.
Hence, the condition \emph{aff}$C_x(\rho)=\mathbb{R}^d$ or aff$C_x(\rho)=\mathbb{R}^{d-1}$  is also necessary for the group $S_x(\rho)$ to be finite.
\end{proof}

\smallbreak
\begin{lemma} Assume $\overline{S}_x(\rho_0)$ and $\overline{S}_x(\rho_0+t)$ are finite groups
as defined at the beginning of the paragraph for clusters
$C_x(\rho_0)$  and $C_x(\rho_0+t)$ respectively. The following statements hold true.

(1) If $S_x(\rho_0+t)=S_x(\rho_0)$, then aff$C_x(\rho_0)=$aff$C_x(\rho_0+t)=\Pi_x^n$, $n\leq d$.

(2) The equality $S_x(\rho_0)=S_x(\rho_0+t)$ is equivalent to  $\overline{S}_x(\rho_0)=\overline{S}_x(\rho_0+t)$.
  \end{lemma}

\begin{proof} (1) Let us prove the first statement by contradiction. Assume that\emph{aff}$C_x(\rho_0)=\Pi_x^{n_1}\subset$aff$C_x(\rho_0+t)=\Pi_x^{n_2}$, $n_1<n_2$.
The group $\overline{S}_x(\rho_0+t)$ of  symmetries operates in $\Pi_x^{n_2}$
where $n_2>n_1$. It is also a finite group,  and the restriction of
any isometry $f'\in \overline{S}_x(\rho_0+t)$ onto plane $\Pi^{n_1}$ is an isometry from
$\overline{S}_x(\rho_0)$. Without loosing  generality we can assume that $n_2$ is equal to $d$, i.e.
$S_x(\rho_0+t)= \overline{S}_x(\rho_0+t)$. It follows from  lemma 4.1 (4) that $S_x(\rho_0+t)$ is a finite group,
and $S_x(\rho_0)$ is an infinite group. Hence, $S_x(\rho_0+t)$ is not equal (proper subgroup) of $S_x(\rho_0)$.
That is a contradiction with the premise of the lemma $S_x(\rho_0+t)=S_x(\rho_0)$.

\smallbreak
(2) Assume that $S_x(\rho_0)=S_x(\rho_0+t)$. It follows from the first part of the lemma that
both of these groups operate in the same plane aff$C_x(\rho_0)=\Pi_x^n$ = aff$C_x(\rho_0+t)=\Pi_x^n$. From the lemma 4.1
 it follows $S_x(\rho_0)= \overline{S}_x(\rho_0)\bigoplus O(x,d-n$ and $S_x(\rho_0+t)= \overline{S}_x(\rho_0+t)\bigoplus O(x,d-n)$.
 Since $S_x(\rho_0)=S_x(\rho_0+t)$, we  conclude that  $\overline{S}_x(\rho_0)= \overline{S}_x(\rho_0+t)$.

Assume now that $\overline{S}_x(\rho_0)= \overline{S}_x(\rho_0+t)$. First of all, it means that they operate in the same planes.
Since by  lemma 4.1, $S_x(\rho)= \overline{S}_x(\rho_0)\bigoplus O(x,d-n)$ and $S_x(\rho_0+t)= \overline{S}_x(\rho_0+t)\bigoplus O(x,d-n)$, we conclude that
 $S_x(\rho_0)=S_x(\rho_0+t)$.
\end{proof}

\smallbreak
Let us remind that according to Definition 2.4 given a $t$-bonded set $X$ in $\mathbb{R}^d$ and  $\rho > 0$,  the $\rho$-cluster $C_x(\rho)$ is
 equivalent to  the $\rho$-cluster $C_{x'}(\rho)$, if there is a space isometry $g$ of $\mathbb{R}^d$,
 such that $g(x)=x'$ and $g(C_x(\rho)) = C_{x'}(\rho)$.

\smallbreak
 \textbf{Statement 4.1.} \emph{Given t-set $X\subset\mathbb{R}^d$ and $\rho_0 > 0$, if clusters  $C_x(\rho_0)$ and $C_{x'}(\rho_0)$ are
  equivalent,  then groups  $S_x(\rho_0)$ and
$S_{x'}(\rho)$ are conjugate.}

\smallbreak \textbf {Proof}. Since $C_x(\rho)$ and $C_{x'}(\rho)$,
i.e. there is an  isometry $g$ of $\mathbb{R}^d$,
 such that $g(x)=x'$ and $g(C_{x}(\rho)) = C_{x'}(\rho)$.
 Let $s\in S_x(\rho)$, then $g \circ s\circ g^{-1}$ maps $C_{x'}(\rho_0)$ onto $C_{x'}(\rho_0)$,
 i.e. $g\circ s \circ g^{-1}\in S_{x'}(\rho_0)$.  We proved that
 $S_{x'}(\rho_0)=gS_x(\rho_0)g^{-1}$.

\smallbreak
\textbf {Statement 4.2.} \emph{Let $X$ be a $t$-bonded set in $\mathbb{R}^d$, and  there is a point $x\in X$  and $\rho_0 > 0$ such that $S_{x} (\rho_0)=S_{x}(\rho_0+t)$.
  If the cluster  $C_x(\rho_0+t)$   is equivalent to a  centered at some  point $x'$ cluster $C_{x'} (\rho_0+t)$, then  $S_{x'} (\rho_0)=S_{x'}(\rho_0+t)$.}

\smallbreak
\textbf {Proof}. Since $C_x(\rho_0+t)$ is equivalent to $C_{x'+t} (\rho_0+t)$, there is an isometry $g$ of $\mathbb{R}^d$
  such that $g(x)=x'$ and $g(C_{x}(\rho_0+t)= C_{x'}(\rho_0+t)$. It follows from the  previous statement and the assumption  $S_{x} (\rho_0)=S_{x}(\rho_0+t)$,
  that $S_{x'}(\rho_0+t)=g\circ S_{x}(\rho_0+t)\circ g^{-1}=g\circ S_{x}(\rho_0)\circ g^{-1}=S_{x'}(\rho_0)$.

Now we'll prove a technical theorem which plays an important
role in  proving  local theorems.

\begin{theorem}[$t$-extension Theorem]
Let in the $t$-bonded set $X$ for two points $x$ and $x'\in X$ and  some
$\rho_0>0$,  clusters $C_x(\rho_0+t)$  and $C_{x'}(\rho_0+t)$ are
equivalent,  and the groups $S_x(\rho_0)$ and $S_x(\rho_0+t)$
coincide:  $$S_x(\rho_0)=S_x(\rho_0+t). \eqno(4.1)$$

Then any isometry $g\in Iso(\mathbb{R}^d)$ such that $g(x)= x'$
and that maps $C_x(\rho_0)$ onto  $C_{x'}(\rho_0)$ (i.e
$g(C_x(\rho_0))=C_{x'}(\rho_0)$) also maps  $(\rho_0+t)$-cluster
$C_x(\rho_0+t)$ onto $C_{x'}(\rho_0+t)$ (i.e $g(C_x(\rho_0+t))=
C_{x'}(\rho_0+t)$.
\end{theorem}

\begin{proof}
By the assumption of the theorem, clusters $C_x(\rho_0+t)$ and
$C_{x'}(\rho_0+t)$ are equivalent. Therefore  there is an isometry
$g\in Iso(d)$ such that $g(x)= x'$ and
$g(C_x(\rho_0+t))=C_{x'}(\rho_0+t)$.

Let $f$ be an arbitrary isometry that maps  $\rho_0$-cluster
$C_x(\rho_0)$ onto $\rho_0$-cluster $C_{x'}(\rho_0)$. Let us take the  composition
$f^{-1} \circ g$. Then we have
$$(f^{-1} \circ g)(C_x(\rho_0))=f^{-1}(g(C_x(\rho_0))=f^{-1}(C_{x'}(\rho_0))=C_x(\rho_0).
\eqno(4.2) $$

From  (4.2) it follows that $f^{-1}\circ g\in S_x(\rho_0)$. Hence,
by  condition (4.1)  of  Theorem 4.1  $f^{-1}\circ g\in
S_x(\rho_0)=S_x(\rho_0 +t)$.  Let us put $f^{-1}\circ g:=s$, $s\in
S_x(\rho_0+t)$. Thus,  $f=g\circ s^{-1}$. Since $g$ maps
$C_x(\rho_0 +t)$ onto $C_{x'}(\rho_0 +t)$ and $s^{-1}$ maps
$C_x(\rho_0 +t)$ onto $C_x(\rho_0 +t)$, we conclude that $f$ maps
$C_x(\rho_0+t)$ onto $C_{x'}(\rho_0+t)$.
\end{proof}

\section{Criterion for Regular t-bonded Systems}
Let us remind that   by  Definition 2.5 a cluster counting
function  $N_(\rho)$ is equal to the cardinal number of equivalence classes of clusters with radius $\rho$
provided the cardinal number is finite.

\smallbreak \textbf {Definition 5.1} (Regular $t$-bonded System).
A $t$-bonded set $X$ is called a \emph{regular t-system} if for
any two points $x$ and $y$ from $X$ there is a symmetry $g$ of $X$
such that $g(x) = y$, i.e. if the symmetry group $Sym(X)$ acts
transitively on $X$.

\smallbreak
\begin{theorem}
Given $t$-set $X$ in $\mathbb{R}^d$, assume that there is $\rho_0
$ such that  the following two conditions hold:
\newline  (1) $N(\rho_0+t)=1$;
\newline (2) for some  point $x_0\in X$
$S_{x_0} (\rho_0)=S_{x_0}(\rho_0+t)$.

\ Then:
\\ (1)  Group  $G\subset Iso(R^d)$ of all symmetries of $X$ operates on $X$ transitively.
\\ (2)  For any point $x\in X$ aff$(C_x(\rho_0))$ = aff$C_x(\rho_0+t)$=aff$X$ = $\mathbb{R}^d$.
\end{theorem}

\begin{proof} First of all,   note that because of  Statement 4.2 and  condition (1) of the theorem
(any two  $(\rho_0+t)$-clusters are equivalent), condition (2) of
the theorem holds true not only  for the point $x_0$, but for any
point $x$ in the set $X$ ($S_x (\rho_0)=S_x(\rho_0+t)$).

Let us prove that the subgroup $G\subset Iso(R^d)$ of all symmetries of $X$  operates on $X$ transitively.

By condition (1) of the theorem for any two points $x$ and $x'$ from $X$, there exists  $g\in Iso(R^d)$ ) such that $g$  maps $C_{x}(\rho_0+t)$ onto
$C_{x'}(\rho_0+t)$  and $g(x)=g(x')$. We'll prove that $g$ maps $X$ onto $X$.

Let us take
an arbitrary point $z\in X$ and connect $x$ to $z$ by a $t$-chain
$x=x_0,x_1, \ldots x_n=z$. We will show that $g$-images of all
points of the chain starting with $x_1$ and ending with $x_n=z$ belong
to $X$.

Since $|xx_1|\leq t$, it follows that  $C_{x_1}(\rho_0)\subset C_{x}(\rho_0+t)$ and
$g(C_{x_1}(\rho_0))=C_{y_1}(\rho_0)$ where $y_1=g(x_1)\in
C_{y_1}(\rho_0+t)\subset X$. By the Theorem 4.1
$g(C_{x_1}(\rho_0+t))=C_{y_1}(\rho_0 +t)$

 Hence we proved that $g(C_{x_1}(\rho_0+t))=C_{y_1}(\rho_0 +t)$ and  $g(x_1)=y_1\in X$.
Since the distance    $|x_ix_{i-1}|\leq t$ for all $i$ such that $1\leq i\leq {n-1}$, applying
the same argument to points $x_i$ and  $x_{i+1}$ as we applied to $x_0$ and  $x_1$, we prove that
for all non negative  integers  $i\leq n-1$,  $g(x_{i+1})=y_{i+1}\in X$ and $g(C_{x_{i+1}}(\rho_0+t))=C_{y_{i+1}}(\rho_0+t)$.

 Hence, $g(z)=g(x_n)=y_n\in X$, and $g(X)\subseteq X$.

 To show that
$g$ is a surjection we note  that the inverse isometry  $g^{-1}$
maps $x'$ onto $x$ and $C_{x'}(\rho)$ onto $C_{x}(\rho_0)$. Applying the same argument to $g^{-1}$ as we aplied to $g$ we show that $g^{-1}$  maps $X$ into $X$. Therefore,  for any $y\in X$,
$g^{-1}y\in X$. Hence,  $g$ is a surjection.

As we already mentioned,  $S_x(\rho_0)=S_x(\rho_0+t)$ for any  point $x\in X$. Therefore, by  lemma 4.2 (part 1)
 aff$C_x(\rho)=\Pi_x^n$ = aff$C_x(\rho_0+t)=\Pi_x^n$, i.e for every  $x\in X$ the following condition  holds
$$
d_x(\rho )= d_x(\rho+t)=n(x). \eqno(3.2)
$$
Then, by Theorem 4.1 $\forall x\in X$ $n(x)\equiv d$, therefore
\emph{aff}$C_x(\rho_0)$=\emph{aff}$C_x(\rho_0+t)$=\emph{aff}$X$=$\mathbb{R}^d$.
\end{proof}

\section{Multi-regular t-bonded Systems: Criterion}
\textbf{Definition 6.1}. A $t$-bonded set $X\subset \mathbb{R}^d$ is a
\emph{multi-regular} $t$-bonded system if there is  a finite set
$X_0=\{x_1,...,x_m\}$ such that
$$X=\cup_{i=1}^k Sym(X)\cdot x_i$$

This  definition is analogous to that of a crystal [Definition
1.6]. However, the situation is quite  different in some respects.
For instance, in  case of a crystal we deal with Delone sets which
are  always infinite sets. Therefore the requirement to represent
a Delone set as a  disjoint union of a finite number of regular
sets determines the selection of  the subclass of Delone sets,  called crystals. In
case of $t$ sets any finite set $X$ is a $t$-bonded set for some value
$t$ and can be thought as a multi-regular system:
$X=\bigcup_{x_i\in X_0}\hbox{\emph{Sym}}(X)\cdot x_i$, where $X_0=X$ and
\emph{Sym}$(X)$ is a trivial group. Nevertheless, the following question
makes sense in any case (finite or infinite) for $t$-bonded sets.

Let us call a $t$-bonded set an \emph{$m$-regular $t$-bonded system} if the number of
classes in $X$/\emph{Sym}$(X)=m$. Are there conditions which guarantee that  a
$t$-bonded set $X$ is  an $m$-regular system? The following criterion
answers the question.

\smallbreak
\begin{theorem} [Local Criterion for $m$-regular $t$-systems]
A $t$-bonded set $X\subset \mathbb{R}^d$ is an $m$-regular
$t$-system if and only if there is some $\rho_0>0$ such that two
conditions hold:
\newline
 1) $N(\rho_0)=N(\rho_0+t)=m;$
\newline
 2) $S_x(\rho_0)=S_x(\rho_0+t), \forall x\in X.$
\end{theorem}

\begin{proof} We precede the formal proof of the theorem  with several remarks and lemma 6.1, which from our point of view,
is  not only technical, but  also has
its own value.
The idea of the proof is similar to that of an
analogous criterion for a crystal ( see, e.g. [13], [14], [15]). Though in  case of the t-set on order to prove
this criterion we do not need to prove that the group \emph{Sym}$(X)$ is a
crystallographic group.

\smallbreak \textbf {Remark 1.} The local criterion for regular
systems [Theorem 5.1] is a particular case of Theorem 6.1. Indeed,
the  condition $N(\rho_0+t)=1$ implies $N(\rho_0)=N(\rho_0+t)=1$.

\smallbreak
\textbf {Remark 2.} Condition 1) of Theorem 6.1 means that with
the increasing radius $\rho$, the number of cluster classes on
segment $[\rho_0,\rho_0+t]$ does not increase, i.e. remains
unchanged: $N(\rho)= N(\rho+t)$. In addition, due to  the
condition  $2)$
, the   cluster group $S_x(\rho_0)$, $\forall x\in
X$, does not get smaller under the $t$-extension of
$\rho_0$-cluster: $S_x(\rho_0)=S_x(\rho_0+t)$.

\smallbreak  The stabilization of these two parameters (the
number of cluster classes and the order of cluster groups) on
segment $[\rho_0,\rho_0+t]$ implies their stabilization on  the half-line  $[\rho_0, \infty)$.

\smallbreak
\textbf {Remark 3.}
The set $X$ can be   represented as a  disjoint union $X=\cup_{i=1}^m X_i$, of not empty subsets $X_i$, where
 $X_i$ is a set of all points of $X$ that are centers of equivalent $\rho_0$-clusters; i.e. $x$ and $x'$ belong to the
 same $X_i$ if and only if there is an  isometry $g\in Iso (d)$ that maps  $C_x(\rho_0)$ onto $C_{x'}(\rho_0)$. With a given $\rho_0$,
 we will call two points that belong to the same $X_i$  $\rho_0$-equivalent points.
It is clear that in a general situation (without any requirements like  condition $1)$ in the theorem's statement, the representation of $X$ as
a disjoint union of subsets $X_i$   is finer for $(\rho_0+t)$-equivalent classes
than for $(\rho_0)$-equivalent classes. However,   condition $1)$ of the theorem means that
 these two representations of $X$ as  unions of $\rho_0$-equivalent and $(\rho_0+t)$-equivalent subsets are the same, i.e $x$ and $x'$
 are $\rho_0$-equivalent if and only if $x$ and $x'$ are  $(\rho_0+t)$-equivalent. We should also note that without any conditions   if $x$ and $x'$  are $(\rho_0+t)$-equivalent,
 then these points are $\rho_0$-equivalent. It  means that in the previous statement written as "if and only if" statement,
condition $1)$ of the theorem actually guarantees that $\rho_0$-equivalence of two points implies $(\rho_0+t)$-equivalence.

\smallbreak
\textbf {Remark 4}. Without losing generality, the condition $2)$ of the theorem could be required not for all points in $X$,
but rather for a finite number of points $X_0=\{x_1,...,x_m\}$ such that $x_i\in X_i$.  By statement 4.2, since all points in each $X_i$ are $(\rho_0+t)$-equivalent,
$S_x(\rho_0)=S_x(\rho_0+t)$ for some $x_i\in X_i$ implies that $S_x(\rho_0)=S_x(\rho_0+t), \forall x\in X_i$.

\smallbreak
\textbf {Remark 5}. It follows from the $t$-extension theorem [ Theorem 4.1]  that if $x$ and $x'$ belong to  $X_i$, then
any isometry $g\in Iso(\mathbb{R}^d)$ that maps $C_x(\rho_0)$ onto  $C_{x'}$ (i.e $g(C_x(\rho_0)=C_{x'}(\rho_0)$) and $x$ onto $x'$
$(g(x)= x')$ also maps  $(\rho_0+t)$-cluster $C_x(\rho_0+t)$ onto $C_{x'}(\rho_0+t)$ (i.e $C_x(\rho_0+t)= C_{x'}(\rho_0+t)$.

\smallbreak
\begin{lemma}  Let a $t$-bonded set  $X$ fulfil conditions $1)$ and  $2)$ of the theorem  and $X_i$ a subset of $X$
of  all  $\rho_0$-equivalent points from $X$,
$i\in [1,m]$. If  $G_i$ is a group  generated by all
isometries $f$ such that $f(x)=x)$ and  $f(C_x(\rho_0))=
C_{x'}(\rho_0)$, $\forall x,x'\in X_i$ ($G_i=<f>$), then:
\newline
1) $G_i$ operates transitively on every set $X_j$,
 $\forall j\in [1,m].$
\newline
 2) The group $G_i$ does not depend on $i$, $G_i=$Sym$(X)$.
\end{lemma}

\begin{proof} Since  for any $i$,  $X_i$ is not empty,  for any two points $x,x' \in X_i$
there is an  isometry g that maps  $C_x(\rho_0)$ onto $C_{x'}(\rho_0)$ and $x$ onto $x'$.
Therefore for any $i$, $G_i$ is not empty. Because of
the way we defined $X_i$,  at least one isometry exists in $G_i$ though   it could be more than one.

To  prove that for
any point $z \in X$, $g(z)\in X$ we can apply the same method  that was used to prove Theorem 5.1,
though  due  to the fact that unlike the conditions of Theorem 5.1, not all points in the set X are $(\rho_0+t)$-equivalent,  and therefore  we must
 be sure that the t-extension theorem  (remark 5) is applicable to the situation under consideration.

Let us take
an arbitrary point $z\in X$ and connect $x$ to $z$ by a $t$-chain
$x=x_0,x_1, \ldots x_n=z$. We will show that $g$-images of all
points in  the chain starting with $x_1$ and ending with $x_n=z$ belong
to $X$.

Since $|xx_1|\leq t$, it follows that  $C_{x_1}(\rho_0)\subset C_{x}(\rho_0+t)$ and
$g(C_{x_1}(\rho_0))=C_{y_1}(\rho_0)$ where $y_1=g(x_1)\in
C_{y_1}(\rho_0)\subset X$. Since $g(C_{x_1}(\rho_0))=C_{y_1}(\rho_0)$ and $y_1=g(x_1)$, it follows that $x_1$ and $y_1$ belong to the same
same set $X_j$. Therefore it  follows from   Theorem 4.1 (see also remark 5) that
$g(C_{x_1}(\rho_0+t))=C_{y_1}(\rho_0 +t)$

Hence,  we proved that $g(C_{x_1}(\rho_0+t))=C_{y_1}(\rho_0 +t)$ and  $g(x_1)=y_1\in X_j\subseteq X$.

Since for any positive
integer  $i \leq m$ the distance    $|x_ix_{i-1} \leq t$, applying
the same argument to the points $x_i$ and  $x_{i+1}$ as we applied to $x_0$ and  $x_1$, we prove that
for any positive
integer  $i \leq m$,  $g(C_{x_{i+1}}(\rho_0+t))=C_{y_{i+1}}(\rho_0+t)$ and  $g(x_{i+1})=y_{i+1}\in X_j\subseteq X$ for some $j\leq m$.

Hence, $g(z)=g(x_n)=y_n\in X$, and $g(X)\subseteq X$.

To show that
$g$ is a surjection,  we notice  that the inverse isometry  $g^{-1}$
maps $x'$ onto $x$ and $C_{x'}(\rho)$ onto $C_{x}(\rho_0)$. Applying the same argument to $g^{-1}$ as we applied to $g$ we show that $g^{-1}$  maps $X$ into $X$. Therefore,  for any $y\in X$
$g^{-1}y\in X$. Hence,  $g$ is a surjection. Therefore $G_i$
is a subgroup of  the group $G:=$Sym$(X)$ (i.e. ($G_i\subseteq G$)

Let us take now any $f \in Sym(X)$, and any point $x\in X_i$. Since $f$ maps $X$ onto $X$. It is clear that $f$
establishes $(\rho_0 +t)$-equivalency of points $x$ and $f(x)$, therefore $f \in G_i$.
Hence,  we proved that   for any positive integer  $i \leq m$,  $G_i$ =$Sym(X)$
\end{proof}

To complete the proof of the theorem we need to make  two observations. First,  by the definition of the set
$X_i$ and group $G_i$,  the group $G_i$ acts transitively on $X_i$, therefore, $X_i=G\cdot x_i$. Second, since for  any positive
integer  $i \leq m$,  $G_i=Sym(X)$, it follows that  $X_i=Sym(X)\cdot x_i$.
 If we denote the set that consists of one point from each $X_i$ by $X_0$ we obtain
 $$X=\bigcup_{x_i\in X_0}\hbox{\emph{Sym}}(X)\cdot x_i.$$
This concludes the proof of the theorem.
\end{proof}

\section{On  t-bonded Sets  of Finite Type and Infinite type}
\textbf{Definition 7.1.} A  set $X$ is said to be \emph{of the finite type} if for each  $\rho>0$ the
number  $N(\rho)$ of classes of $\rho$-clusters is finite.

\smallbreak
It is easy to see that for any uniformly discrete  set $X$ the  function $N(\rho)$
is always defined and equal to 1 for all $\rho<r$. It is not hard
to prove the following:

\smallbreak
\textbf{Statement 7.1}. \emph{If $X$ is a Delone set with $N(2R)<
\infty $, then for all $\rho>0$ the cluster counting function $N(\rho)<\infty$, i.e. $X$ is a set of the finite type.}

\begin{proof}The key reason for  this fact  is as follows. Given the  Delone set
$X$, we can construct the  Delone tiling corresponding to the
Delone set. Let us take a ball $B_x(\rho)$ centered at point $x\in
X$. Then the Delone tiles which overlap with the ball form a
\emph{pavement}  of the ball. The $\rho$-cluster $C_x(\rho)$ is
obviously a subset of all those vertices of the pavement of the
ball $B_x(\rho)$ which are located in the ball.

Now we take a point $z\in \mathbb{R}^d$ and consider the family
$\cal P$ of all possible face-to-face pavements $P$ of the ball
$B_{z}(\rho)$ by tiles with the following conditions:
\newline a) a tile of  pavement $P$  is congruent
to a  tile from the Delone tiling for the set $X$;
\newline b) the center $z$ of the ball $B_z(\rho)$ is a vertex of the pavement $P$.

We emphasize  that we do not assume  that any  pavement is
congruent
 to a fragment of the Delone tiling for $X$. On the
other hand, it is obvious that for any $x\in X$ a pavement of the
ball $B_x(\rho)$ which is a fragment of the Delone tiling for $X$
 is congruent to some pavement $P\in \cal P$. Moreover, the
cluster $C_x(\rho)\subset X$ is congruent to a set of those
vertices of the  $P$ which belong to the ball $B_z(\rho)$.

We note that if in the family $\cal P$ there are just finitely
many non-congruent pavements, then in $X$  for a given $\rho>0$
there are also finitely many non-equivalent $\rho$-clusters.  Now
we show that the family  $\cal P$  is finite.

In fact, the condition $N(2R)<\infty$ implies that in the Delone
tiling for the set $X$  there are just finitely many pairwise
non-congruent Delone tiles. It is known  that  the Delone tiling
is a face-to-face tiling. Assume  two  tiles $Q$ and $Q'$ have a
congruent  hyperface. It is easy to see that the polytope $Q$ can
be put to the polytope $Q'$ by a common hyperface only in a finite
number of ways. Due to the two conditions (finiteness of classes
of Delone tiles for $X$ and finiteness of non-congruent pairs
$(Q,Q')$ of tiles adjacent on a common hyperface),  only finitely
many face-to-face pavements $P$ of the ball $B_z(\rho)$ with the
above-mentioned properties a) and b) can be  constructed. It
follows from this that $N(\rho)<\infty$.\endproof
\end{proof}

\smallbreak
 Now we return to a more general  case when $X$ is a $t$-bonded set.
 The following theorem shows that  the case of $t$-bonded sets
 differs from the case of Delone sets.

\smallbreak
\begin{theorem}
Given $t>0$, for an arbitrarily large  $k>0$ there is a $t$-bonded set
$X\subset \mathbb{R}^2$ such that $N(kt)<\infty$, but  for any
$\varepsilon>0$ $N(kt+\varepsilon)=\infty$. Thus the above mentioned $t$-bonded set $X$  is not a set of the finite
type.
\end{theorem}

\begin{proof}
We will construct an example of such set $X$.

Let us  take two positive numbers $a$ and $b$ such that  $a/b$ is
irrational and  $t< a, b< 2t$, where $t>0$ is a given parameter
for $X$. Assume that $L_1:$ $v=0$ and $L_2:$ $v=kt$ are  two horizontal lines in the plane $(u,v)$.

Along each of these lines we construct a "horizontal" broken line
whose vertices will be a subset of the  set $X$.

The  set $X_1$ of vertices of the first broken line is determined  by
the formulas:

\noindent $ x_i=(ia/2,0)$, if  $i$ is even and

\noindent $ x_i= (ia/2,-\theta_1)$ if $i$ is odd, $\theta_1>0,$
$i\in Z$.

\smallbreak
 The altitude $\theta _1$ in an equilateral triangle $\triangle x_{i-1}x_{i}x_{i+1}$ is chosen so
that side-lengths $|x_0x_1|$, $|x_1x_2|$, ...  of the broken
line are equal to $t$.

\smallbreak
The set $X_2$ of  vertices  of the second horizontal broken line
along the line $L_2$ is determined  by
the formulas:

\noindent $ y_i=(ib/2, kt)$, if  $i$ is even and

\noindent $ y_i= (ib/2 ,kt+\theta_2 )$ if $i$ is odd,
$\theta_2>0$, $i\in Z$.

\smallbreak The altitude $\theta _2$ in the equilateral
triangle $\triangle y_{i-1}y_{i}y_{i+1}$ is chosen so that sides
$|y_0y_1|$, $|y_1y_2|$, ... of the broken line are equal to $t$.

Note that all points of $X_1$ with even indices are on the line
$L_1$, and all points of $X_1$ with odd indices are below this
line.

Similarly   all points of $X_2$ with even indices are on the
line $L_2$, and all points of $X_2$ with odd indices are above  this
line.

The distance between $(0,0)$ and $(0,kt)$ is equal to $kt$. At
the same time, the distance between any other pairs of points
from $X_1$, and $X_2$ (but the pair  $(0,0)$ and $(0,kt)$) is greater than $kt$.
It is clear that for the pair $x_i \in X_1$ and $y_j \in X_2$ the distance is greater
than  $kt$ provided either $i$ or $j$ is odd. Assume that $i$ and $j$ are both even. If
$|x_iy_j|=kt$, then the interval that connects $x_i$ and $y_j$ is parallel to the interval
that connects $(0,0$ and $(0,kt)$, which  in turn implies that $ia/2=jb/2$, and  $b/a$ is rational.
However,  by  choice of $b$ and $a$ the ratio $b/a$ is irrational.

To complete the construction of the set $X$ we  add to the sets
$X_1$ and $X_2$ the third set $X_3$ which is described below.

Let $X_3$ be a set of vertices of a broken line along the interval
that connects $(0,0)$ to $(0,kt)$. Below  we explain how this broken
line is constructed.

Let $c>0$ be such that $t\leq c<2t$ and $c$ divides $kt$:
$\frac{kt}{c}=n\in Z$. Construct the following   broken line with
vertices $z_0, z_1, z_{2n}$ and links   all equal to $t$.

\noindent  $z_i=(0, c i/2)$ if $i$ is even, $z_0=x_0$   and
$z_{2n}=y_0$;

\noindent $z_i=(\theta_3, ci/2)$ if $i$ is odd.

Here  the altitude $\theta_3>0$ in the equilateral triangle
$z_{i-1}z_iz_{i+1}$ is chosen so that the lateral sides
$z_{i-1}z_i$ and $z_iz_{i+1}$ are both equal to $t$.

With the construction of  $X_3$ we completed constructing  the
$t$-bonded set $X$ which is defined as  $X=X_1\cup X_2\cup X_3$.

\vskip -0.3cm

\begin{center}
\begin{figure}[!ht]
\hskip 2cm{ \hbox{\includegraphics[width=8cm]{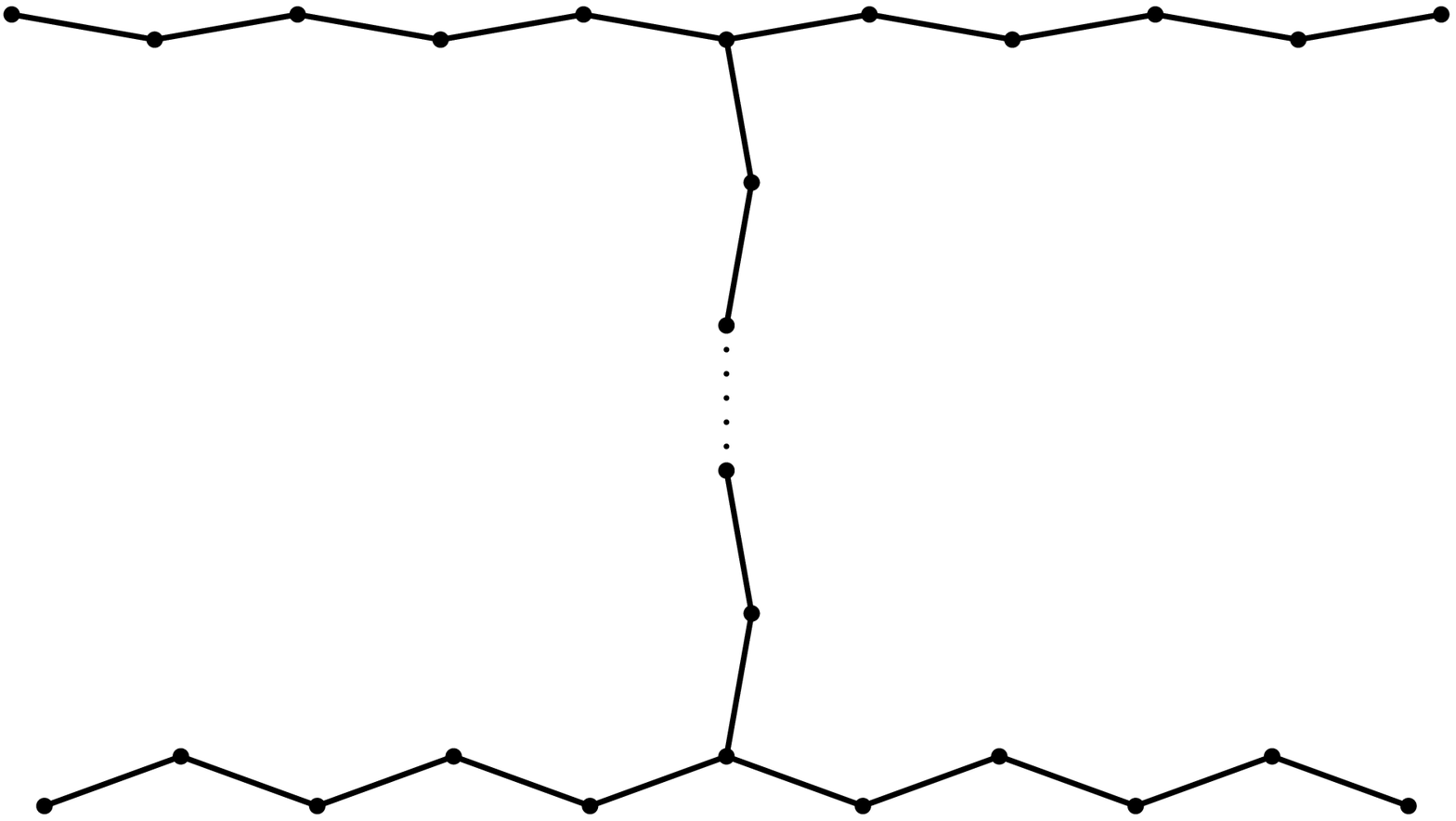}
\raise4.0cm\llap{$y_0$\kern4.2cm}
\raise4.6cm\llap{$y_1$\kern3.1cm}
\raise4.0cm\llap{$y_2$\kern2.5cm}
\raise0.9cm\llap{$z_1$\kern4.0cm}
\raise3.5cm\llap{$z_{2n-1}$\kern3.6cm}
\raise-0.1cm\llap{$x_0$\kern4.6cm}
\raise-0.2cm\llap{$x_1$\kern3.8cm}
\raise0.5cm\llap{$x_2$\kern3.3cm} \raise-1.3cm\llap{Fig. 1. Point
set $X=X_1\cup X_2\cup X_3$ \kern2.2cm} }}
\end{figure}
\end{center}

\vskip -0.5cm

 The role of $X_3$ is to  make the entire set
$X$ $t$-connected. We emphasize that by  construction of the set
$X$, any $t$-cluster in the set $X$ has  rank 2, i.e. the
affine hull with dimension 2.

It is not hard to prove  that $N(kt)<\infty$. In fact, all clusters
$C_{x_i}(kt)$, $x_i\in X_1$ are equivalent if $\frac{|i|}{2}a>kt
+\theta_3$ that holds when $|i|>\frac{2(kt+\theta_3)}{a}$.

Analogously, for points $y_j\in X_2$ all clusters $C_{y_j}(kt)$
are equivalent if $|j|>\frac{2(kt+\theta_3)}{b}$.

Besides these two classes of $kt$-clusters in $X$,  there is a
finite amount of $kt$-clusters centered at points $x_i$ and $y_j$
with $|i|\leq \frac{2(kt+\theta_3)}{a}$ and $|j|\leq
\frac{2(kt+\theta_3)}{b}$ respectively, and there is a finite
amount of $kt$-clusters centered at points of  $X_3$. Therefore
$N(kt)<\infty$.

Throughout the rest of the proof we will consider clusters
centered at points $x_i\in X_1$ and $y_j\in X_2$
 only with even indices $i$ and $j$. To be consistent with all the notations let us
redesignate the points $x_i$ with even indices $i$ with new natural indices $p$,
 where $p=i/2$, $x_i=x_p$. A  similar change  is  done for the points  $y_j$ where $j$ is even, $j\mapsto q$,
 where $q=j/2$. Then  $x_p=(pa,0)$ and  $y_q=(qb,kt)$,
 where $p$ and $q$ are positive integer numbers.

Let us  take $\varepsilon>0$ so small that the ball
$B_{x_p}:=B_{x_p}(kt +\varepsilon)$  and $L_2$ intersect over a
chord with the the length $2\delta (\varepsilon) $ where
$\delta<b/4$. Since $\delta<b/4$,   $|B_{x_{p}}(kt +\varepsilon)\cap X_2|\leq 1$
for any $p\in \mathbb{N}$.

Let us consider a set $Q$ of all pairs $(p,q)\in \mathbb{N}^2$
such that $B_{x_p}\cap L_2\ni y_q$, i.e the point $y_q$ belongs to
the cluster $C_{x_p}(kt +\varepsilon)$-cluster centered in the
point $x_p$. This is equivalent to the inequality
 $|pa-qb|<\delta$, or to the inequality
$$|\alpha-\frac{p}{q}|< \frac{\delta/a}{q},
\hbox{ where } \alpha=\frac{b}{a} \hbox{ is irrational}. \eqno(7.1)
$$

\smallbreak By the Dirichlet  theorem, the  inequality
$$|\alpha-\frac{p}{q}|< \frac{1}{q^2}  \eqno(7.2)$$
has infinitely many solutions in positive integers $(p,q)$.

Moreover,  let $\alpha=[a_0;a_1,\ldots,a_n,\ldots]$  be the
continued fraction of $\alpha$ and
$\frac{p_n}{q_n}=[a_0;a_1,\ldots,a_n]$ the $n$-th convergent. The
sequence of the convergents determines the following two sequences
of positive integer numbers
$$q_1<q_2<q_3<\ldots \hbox{ and } p_1<p_2<p_3<\ldots \eqno(7.3)$$
such that
$$|\alpha-\frac{p_n}{q_n}|< \frac{1}{q_nq_{n+1}}<\frac{1}{q_n^2} \eqno(7.4)  $$
 Due to (7.4)  for all large enough  $n$ such that $q_{n+1}>1/\delta$ we have
$$ |\alpha -\frac{p_n}{q_n}|<\frac{\delta}{q_n}. \eqno(7.5)
$$
Now for all these values of $n$ we take the following  pairs of points
 $(x_{p_n},y_{q_n})$ where  $x_{p_n}=(p_na, 0)\in X_1$ and
$y_{q_n}=(q_nb, kt)\in X_2$. Due to (7.5) the following inequality
holds
 $$|q_nb-p_na|<q_n\frac{\delta}{q_n}\leq\delta. \eqno(7.6)$$

 \smallbreak By choice of $\delta=\delta(\varepsilon)$ it follows from (7.6) that $y_{q_n}$ belongs to the cluster
 $C_{x_{p_n}}(kt+\varepsilon)$. Therefore,  $(p_n,q_n)\in Q$
 for all $n>N_0(\delta)$.

Let  $(p,q)$ and $(p', q')$ be two pairs of indices   and
$x_p=(pa,0)$, $y_q=(qb,kt)$ and $x_{p'}=(p'a,0)$,
$y_{q'}=(q'b,kt)$ pairs of corresponding points. If the pairs
$(p,q)$ and $(p'q')$ both belong to $Q$, then $y_q \in C_{x_p}(kt
+\varepsilon)$ and $y_{q'} \in C_{x_{p'}}(kt +\varepsilon)$. We
show  that for the clusters $C_{x_p}(kt +\varepsilon)$ and
$C_{x_{p'}}(kt +\varepsilon)$ there are at most two potentially
possible isometries of the plane that establish equivalence of
these clusters.

Since $kt +\varepsilon\geq 2t+\varepsilon>a$,
 points $x_{p-1}=((p-1)a,0)$ and $x_{p+1}=((p+1)a,0)$ also belong to the cluster $C_{x_{p}}(kt +\varepsilon)$.
 We can even say that these two points are the only two points that belong to the
cluster at the distance $a$ from the center of the cluster
$C_{x_p}(kt+\varepsilon)$, if $x_p\neq x_0$. The same observation
is true regarding the cluster $C_{x_{p'}}(kt +\varepsilon)$,
centered in the point $x_{p'}$. It means that any isometry of
the plane that establishes equivalence of the clusters
$C_{x_{p}}(kt +\varepsilon)$ and $C_{x_{p'}}(kt +\varepsilon)$,
should map the center $x_{p}$ onto the $x_{p'}$, and  points
$x_{p-1}$ and $x_{p+1}$ onto points $x_{p'-1}$ and $x_{p'+1}$
(though we don't claim  that the order of these two points should
be preserved). There could be only two isometries of this type. The
first one is a parallel shift when $x_{p-1}$ maps onto
$x_{p'-1}$, and $x_{p+1}$ maps onto $x_{p'+1}$(the isometry
preserves the order of the points $x_{p-1}$  and $x_{p+1}$). The
second one, when $x_{p-1}$ maps onto $x_{p'+1}$, and $x_{p+1}$
maps onto  $x_{p'-1}$, is a parallel shift followed by   the
reflection around the line $u=ap'$.

\smallbreak
Let us  prove now that $C_{x_p}(kt +\varepsilon)$ and
$C_{x_{p'}}(kt +\varepsilon)$ are not equivalent if $(p,q)$ and
$(p'q')\in Q$. Note,  that $y_q \in C_{x_p}(kt +\varepsilon)$ and
$y_{q'} \in C_{x_{p'}}(kt +\varepsilon)$. We will prove that neither
a parallel shift, nor  a parallel shift followed by
reflection around line $u=p'a$, can establish equivalence of the
clusters $C_{x_{p}}(kt +\varepsilon)$ and $C_{x_{p'}}(kt
+\varepsilon)$.

 \smallbreak Assume that for some real number l (positive or negative)$bq =ap-l$. In case of  parallel  shift
  $bq' =ap'-l$,  and
 $bq-bq' =ap-ap'$, hence, $qb - q'b = pa- p'a$, i.e. $b/a=(p-p')/(q-q')$. If the isometry is the shift and the reflection,
 as described above, then
 $bq' =ap'+l$, hence $qb + q'b = pa+ p'a$, or $b/a=(p+p')/(q+q')$. In either case $b/a$ is a rational number
  that contradicts the choice of $a$ and $b$.

 \smallbreak
Therefore, $(kt +\varepsilon)$-clusters $C_{x_p}(kt +\varepsilon)$
and $C_{x_p'}(kt +\varepsilon)$, are not equivalent for any
different $p$ and $p'$ from the infinite sequence
$p_1<p_2<p_3<\ldots $. Hence, the set X is a $t$-bonded set of infinite
type. \end{proof}

\smallbreak
\section {Summary}
In the paper we  developed the local theory of regular and
multi-regular $t$-bonded sets.  The significance of this theory is
that the terms of $t$-bonded sets seems to be  more appropriate for
describing the chemical bonds existing  between atoms in real
structures.  In many respects this theory follows
in the tracks of the local theory of regular Delone systems. However,
the  $t$-bonded sets essentially extend the limits of the family of
Delone sets,  and therefore it is no surprise that    in spite of the
similarity  of the theories,  there are essentially new features in the
behavior of the $t$-bonded sets that are not the  Delone sets.

From our point of view the studies of $t$-bonded sets should be continued in two
directions. First,  to get the upper bound for the radius
$\rho_0$ of a cluster such that the  condition  $N(\rho_0)=1$
implies regularity of a $t$-bonded set in the 3D space. Second, regarding potential application  of the   theory,  it makes sense  to extend
the above mentioned theory of regular sets for clusters defined
by other metrics.

\section*{References}.

\medskip

\noindent [1] N.P.~Dolbilin, On Local Properties for Discrete Regular Systems,
Uspekhi Matem. Nauk, 1976, v.230, N.3,  516–519 (in Russian).

\noindent [2] B.~Delaunay, Sur la sph\`{e}re vide. A la
m\'{e}moire de Georges Vorono\"{\i}. Bulletin de l'Acad\'{e}mie
des Sciences de l'URSS. 1934, Issue 6,  Pages 793–800.

\noindent [3] B.N.~Delone, Geometry of positive quadratic forms,
Uspekhi Matem. Nauk, 1937, 3,  16–62 (in Russian).

\noindent [4] A.~Schoenflies, Kristallsysteme und
Kristallstruktur, Leipzig,  1891 - Druck und verlag von BG Teubner

\noindent [5] L.~Bieberbach, \"Uber die Bewegungsgruppen des
n-dimensionalen Euklidischen R\"aumes I, Math. Ann. 70 (1911),
207-336; II, Math. Ann. 72 (1912), 400-413.

\noindent [6] E.S.~Fedorov, Elements of the Study of Figures,
Zap. Mineral. Imper. S.Peterburgskogo Obschestva, 21(2), 1985,
1-279.

\noindent [7] B.N.~Delone, N.P.~Dolbilin, M.I.~Stogrin,
R.V.~Galiuilin, A local criterion for regularity of a system of
points. Soviet Math. Dokl., 17, 1976, 319-–322.

\noindent [8] N.P.~Dolbilin, M.I.~Shtogrin, A local criterion for
a crystal structure, Abstracts of the IXth All-Union Geometrical
Conference, Kishinev, 1988, p. 99 (in Russian).

\noindent [9] N.P.~Dolbilin, The extension theorem, Discrete
Math., 221:1-3, Selected papers in honor of Ludwig Danzer (2000),
43--59.

\noindent [10] M.I.~Stogrin,   On the upper bound for the
order of axis of a star in a locally regular Delone set.
Geometry, Topology, Algebra and Number Theory, Applications.
The International Conference dedicated to the 120-th anniversary
of Boris Nikolaevich Delone (1890-1980) (Moscow, August 16-20, 2010),
Abstracts, Steklov Mathematical Institute, Moscow, 2010, 168-169 (in Russian).

\noindent [11] Mikhail M. ~Bouniaev, Nikolay P. ~Dolbilin, Oleg R.~Musin, Alexey S.~Tarasov,
Geometrical Problems  Related to Crystals, Fullerence, and Nanoparticles Structure.
In Forging Connections Between Computational Mathematics  and Computational Geometry (Ke Chen, editor),
Springer International Publisher in Mathematics and Statistics, 2016, pp 139-152.

\noindent [12] Mikhail M. ~Bouniaev, Nikolay P. ~Dolbilin, Oleg R.~Musin, Alexey S.~Tarasov,
Two Groups of  Geometrical Problems Related to Study of Fullerenes and  Crystals.
Journal of Mathematics, Statistics and Operations Research(JMSOR) Vol.2 No.2(Print ISSN: 2251-3388,
E-periodical: 2251-3396) pp.18-28

\noindent [13] Mikhail M. ~Bouniaev, Nikolay P. ~Dolbilin, Local Theory of Crystals: Development and Current Status.
 Proceedings, 4-th Annual International Conference  on Computational Mathematics, Computational Geometry and Statistics,
 Singapore,  26th-27thJanuary, 2015, pp. 39-45.

\noindent [14] N.P.~Dolbilin, J.C.~Lagarias, M.~Senechal,
Multiregular point systems. Discr. and Comput. Geometry, 20, 1998,
477–498.

\noindent [15] N.P.~Dolbilin, A Criterion for   crystal and locally
antipodal Delone sets. Vestnik Chelyabinskogo Gos. Universiteta,
2015, 3 (358), 6-17 (in Russian).

\noindent [16] N.P.~Dolbilin, A.N.~Magazinov, Locally antipodal
Delauney Sets, Russian Math. Surveys.70:5 (2015), 958-960.

\noindent [17] N.P.~Dolbilin, A.N.Magazinov, The Uniqueness
Theorem for Locally Antipodal Delone Sets, ``Modern Problems of
Mathematics, Mechanics and Mathematical Physics'',  II, Collected
papers, Steklov Institute Proceedings, 294, MAIK, M., 2016 (in
print).


\end{document}